\documentclass{amsart}
\usepackage{amssymb,amsmath,latexsym}
\usepackage{amsthm}
\usepackage{fontenc}
\usepackage{amssymb}
\numberwithin{equation}{section}

\newtheorem{theorem}{Theorem}[section]

\newtheorem{lemma}[theorem]{Lemma}

\begin{document}
\author{Alexander E Patkowski}
\title{Some Remarks on Glaisher-Ramanujan Type Integrals}

\maketitle
\begin{abstract}Some integrals of the Glaisher-Ramanujan type are established in a more general form than in previous studies. As an application we prove some Ramanujan-type series identities, as well as a new formula for the Dirichlet beta function at the value $s=3.$\end{abstract}



\section{Introduction}
In a paper of Glasser [5], the work of Glaisher [4] and Ramanujan [7] was extended to present further evalutions of the integral
\begin{equation}\int_{0}^{\infty}\eta^{n}(ix)f(x)dx,\end{equation}
for integers $n \ge 1,$ and particular elementary functions $f(x).$ Here, as usual, the Dedekind-eta function is given by
\begin{equation} \eta(ix)=q^{1/24}\prod_{n\ge1}^{\infty}(1-q^n),  \end{equation}
where $q = e^{-2\pi x},$ for real $x > 0.$ For some commentary on integrals of this type we refer the reader to [5, 7] and references therein.
\par In this note, we restrict our attention to the case $f(x) = e^{-bx} \cos(cx),$ and $f(x) = e^{-bx} \sin(cx),$ in (1.1). This provides a refinement to integrals like (8) of [5] and (19)--(28) in [4]. That is, we shall prove

\begin{theorem} For $b>0$ and $c>0$ we have
\begin{equation}\int_{0}^{\infty}\eta^{3}(i4x/\pi)e^{-b^2 x}\sin(cx)dx=\frac{\pi}{4c}\frac{\sinh(\frac{\pi}{2}A(b,c))\sin(\frac{\pi}{2}B(b,c))}{\sinh^2(\frac{\pi}{2} A(b,c))+\cos^{2}(\frac{\pi}{2}B(b,c))},\end{equation}
\begin{equation}\int_{0}^{\infty}\eta^{3}(i4x/\pi)e^{-b^2 x}\cos(cx)dx=\frac{\pi}{4}\frac{\cosh(\frac{\pi}{2}A(b,c))\cos(\frac{\pi}{2}B(b,c))}{\cosh^2(\frac{\pi}{2} A(b,c))-\sin^{2}(\frac{\pi}{2}B(b,c))},\end{equation}
where $2A(b,c)^2=\sqrt{b^4+c^2}+b^2,$ and $2B(b,c)^2=\sqrt{b^4+c^2}-b^2.$
\end{theorem}
Note that $c\rightarrow0$ of (1.4) gives (14) of [5].
\par From here we can easily extend the work of Glasser to obtain evaluations of integrals involving $\eta^{n}(iax)\eta^{k}(ibx),$ for integrals $n, k\ge1,$ and
$a,b\in\mathbb{R}.$ Throughout this paper we define
\begin{equation}\chi(n)=\begin{cases} 0,& \text {if } n=0\pmod{2},\\ 1, & \text{if } n=1\pmod{4}, \\ -1, & \text{if } n=3\pmod{4}.\end{cases}\end{equation}

\begin{theorem} For $c>0$ we have
\begin{equation}\int_{0}^{\infty}\eta^{6}(i4x/\pi)\sin(cx)dx=\frac{\pi}{2c}\sum_{n\ge1}\chi(n)n\frac{\sinh(\frac{\pi}{2}A(n,c))\sin(\frac{\pi}{2}B(n,c))}{\sinh^2(\frac{\pi}{2} A(n,c))+\cos^{2}(\frac{\pi}{2}B(n,c))},\end{equation}
\begin{equation}\int_{0}^{\infty}\eta^{6}(i4x/\pi)\cos(cx)dx=\frac{\pi}{4}\sum_{n\ge1}\chi(n)n\frac{\cosh(\frac{\pi}{2}A(n,c))\cos(\frac{\pi}{2}B(n,c))}{\cosh^2(\frac{\pi}{2} A(n,c))-\sin^{2}(\frac{\pi}{2}B(n,c))}.\end{equation}
\end{theorem}

Unfortunately, as we observe, the right sides of (1.4) and (1.5) are not expressible in terms of elementary functions like those of Glasser's [5] and Glaisher's [4]. We can also obtain other examples using the same procedure as Glasser, by appealing to different theta series. In particular, by Euler's identity [1, p.575] we have the following.

\begin{theorem} For $c>0$ we have
\begin{equation}\int_{0}^{\infty}\eta^{3}(i4x/\pi)\eta(i12x/\pi)\sin(cx)dx=\frac{\pi}{4}\sum_{n\in\mathbb{Z}}(-1)^n\frac{\cosh(\frac{\pi}{2}A(6n+1,c))\cos(\frac{\pi}{2}B(6n+1,c))}{\cosh^2(\frac{\pi}{2} A(6n+1,c))-\sin^{2}(\frac{\pi}{2}B(6n+1,c))}.\end{equation}
\end{theorem}

\section{The Proof}
To prove Theorem 1.1 we require some simple series evaluations that we were unable to find in the literature. Our methods are similar to those of [4] and we only require some known integral evaluations and the Poisson summation formula for Fourier sine transforms [2, p.257]. If $f(x)$ is a continuous, real-valued function with bounded total variation on $[a, b]$ then
\begin{equation}\sum_{a\le n \le b}\chi(n)f(n)=\sum_{n\ge1}\chi(n)\int_{a}^{b}f(x)\sin(\pi xn/2)dx.\end{equation}
By I. S. Gradshteyn and I. M. Ryzhik [6, p.428], we have
\begin{equation}\int_{0}^{\infty}\frac{x\sin(ax)dx}{(x^2+b^2)^2+c^2}=\frac{\pi}{2c}e^{-aA(b,c)}\sin(aB(b,c)),\end{equation}
\begin{equation}\int_{0}^{\infty}\frac{x(x^2+b^2)\sin(ax)dx}{(x^2+b^2)^2+c^2}=\frac{\pi}{2}e^{-aA(b,c)}\cos(aB(b,c)),\end{equation}
where $A(b,c)$ and $B(b,c)$ are as in Theorem 1.1, and $a>0,$ $b>0,$ and $c>0.$
\begin{lemma}For $b>0$ and $c>0$ we have
\begin{equation}\sum_{n\ge1}\chi(n)\frac{n}{(n^2+b^2)^2+c^2}=\frac{\pi}{4c}\frac{\sinh(\frac{\pi}{2}A(b,c))\sin(\frac{\pi}{2}B(b,c))}{\sinh^2(\frac{\pi}{2} A(b,c))+\cos^{2}(\frac{\pi}{2}B(b,c))},\end{equation}
\begin{equation}\sum_{n\ge1}\chi(n)\frac{n(n^2+b^2)}{(n^2+b^2)^2+c^2}=\frac{\pi}{4}\frac{\cosh(\frac{\pi}{2}A(b,c))\cos(\frac{\pi}{2}B(b,c))}{\cosh^2(\frac{\pi}{2} A(b,c))-\sin^{2}(\frac{\pi}{2}B(b,c))}.\end{equation}
\end{lemma}

\begin{proof} For (2.4) apply (2.1) with $f(x)=\frac{x}{(x^2+b^2)^2+c^2}$ and invoke (2.2). For (2.5) apply (2.1)
with $f(x)=\frac{x(x^2+b^2)}{(x^2+b^2)^2+c^2}$ and invoke (2.3).
\par For (1.3), we use identity (12) of [5] (with $x$ replaced by $x4/\pi$) to find
$$\int_{0}^{\infty}\eta^{3}(ix4/\pi)e^{-b^2x}\sin(cx)dx$$
$$=\sum_{n\ge1}\chi(n)n\int_{0}^{\infty}e^{-(b^2+n^2)x}\sin(cx)dx$$
$$=c\sum_{n\ge1}\chi(n)\frac{n}{(n^2+b^2)^2+c^2}.$$
\end{proof}
By (2.4) of Lemma 2.1 the proof is straightforward. It is not difficult to prove (1.4).
The only difference is that we appeal to the Fourier cosine transform and employ (2.5).

\section{An Application to Ramanujan-Type Series}
In Ramanujan's notebook [2] we find the amazing formula for $\zeta(\frac{1}{2}):$ If $x >0,$ then
$$\sum_{n\ge1}^{\infty}\frac{1}{e^{n^2x}-1}=\frac{\pi^2}{6x}+\frac{1}{2}\sqrt{\frac{\pi}{x}}\zeta(\frac{1}{2})+\frac{1}{4}$$
\begin{equation}+\sqrt{\frac{\pi}{2x}}\sum_{n\ge1}\frac{1}{\sqrt{n}}\left(\frac{\cos(\frac{\pi}{4}+2\pi\sqrt{\pi n/x})-e^{-2\pi\sqrt{\pi n/x}}\cos(\frac{\pi}{4})}{\cosh(2\pi\sqrt{\pi n/x})-\cos(2\pi\sqrt{\pi n/x})}\right).\end{equation}

Several authors have produced generalizations of this formula [2, 3, 8, 9]. Authors [9] obtain a formula for the Dirichlet $L$-function for $\bar{\chi},$ the primitive Dirichlet character modulo $q,$ at $s =1.$ In this section we will obtain a formula for the special value $s=3$
of the Dirichlet beta function [1]
\begin{equation} \beta(s)=\sum_{n\ge0}\frac{(-1)^n}{(2n+1)^s}.\end{equation}
\begin{theorem} For $z>0$ we have
$$\frac{\pi}{8}+\sum_{n\ge1}\frac{\chi(n)}{n(e^{n^2z}-1)}$$
\begin{equation}=\frac{\beta(3)}{z}+\frac{1}{2\pi}\sum_{n\ge1}\frac{\sinh(\frac{\pi}{2}\sqrt{\frac{n\pi}{ z}})\sin(\frac{\pi}{2}\sqrt{\frac{n\pi}{ z}})}{n(\cosh(\pi\sqrt{\frac{n\pi}{ z}})+\cos(\pi\sqrt{\frac{n\pi}{ z}}))}.\end{equation}
\end{theorem}
\begin{proof} Under the same hypothesis as for sine transforms for the Poisson summation formula, we have [2, p.252]
\begin{equation}\sum_{a\le n \le b}'f(n)=\int_{a}^{b}f(x)dx+2\sum_{n\ge1}\int_{a}^{b}f(x)\cos(\pi 2xn)dx,\end{equation}
with the additional condition that the prime on the sum indicates only $\frac{1}{2}f(a)$ is counted if $a$ is finite, and similarly for $b.$
We choose the function ($x, z>0$) $$f(x)=\sum_{n\ge1}\chi(n)\frac{e^{-n^2xz}}{n},$$
which has the range $\mathbb{R}^{+}=(0,\infty).$ $f(x)$ has bounded variation since, over any closed interval $I\subset\mathbb{R}^{+},$ there exists a constant $M$ such that $\sum_{i\ge1}^{n}|f(x_{i})-f(x_{i-1})|<M,$ for all partitions of $I.$
\par Glaisher [4, eq.(23)] offers
\begin{equation}\sum_{n\ge1}\chi(n)\frac{e^{-n^2z}}{n}=\frac{1}{2}\int_{0}^{\infty}\cos(zx)\frac{\sinh(\frac{\pi}{2}\sqrt{\frac{x}{2}})\sin(\frac{\pi}{2}\sqrt{\frac{x}{ 2}})dx}{x(\cosh(\pi\sqrt{\frac{x}{2}})+\cos(\pi\sqrt{\frac{x}{2}}))},\end{equation}
or
\begin{equation}\int_{0}^{\infty}\cos(zx)\sum_{n\ge1}\chi(n)\frac{e^{-n^2z\alpha}}{n}dz=\frac{1}{2}\frac{\sinh(\frac{\pi}{2}\sqrt{\frac{x}{2\alpha}})\sin(\frac{\pi}{2}\sqrt{\frac{x}{ 2\alpha}})}{x(\cosh(\pi\sqrt{\frac{x}{2\alpha}})+\cos(\pi\sqrt{\frac{x}{2\alpha}}))},\end{equation}
for $\alpha>0.$ Choosing $a=0$ and $b=\infty$ in (3.4) with our choice of $f(x),$ we get the theorem after noting $$\int_{0}^{\infty}f(x)dx=\frac{\beta(3)}{z},$$ and $f(0)=\frac{\pi}{4}.$
\end{proof}
\par Here we were able to observe some further integral evaluations and an interesting application. It remains a challenge to the reader to find specific instances when integrals of the type (1.1) involving $\eta^{n}(aix)\eta^{k}(bix)$ can be evaluated for natural numbers $n, k\ge1.$

1390 Bumps River Rd. \\*
Centerville, MA
02632 \\*
USA \\*
E-mail: alexpatk@hotmail.com
\end{document}